\documentclass[11pt]{article}
\usepackage[a4paper,margin=1in]{geometry}
\usepackage[T1]{fontenc}
\usepackage{lmodern}
\usepackage{microtype}
\usepackage{amsmath,amssymb,amsthm,mathtools}
\usepackage{enumitem}
\usepackage{booktabs}
\usepackage{array}
\usepackage{hyperref}
\usepackage[nameinlink,capitalise]{cleveref}
\usepackage{xcolor}
\usepackage{fancyhdr}
\usepackage{tikz}

\hypersetup{
  colorlinks=true,
  linkcolor=blue!55!black,
  citecolor=blue!55!black,
  urlcolor=blue!55!black,
  pdftitle={A Chain- and Diagram-Level Semantics for Morphological Calculus},
  pdfauthor={Baruch Schneider, Diana Schneiderova (Barseghyan), Yifan Zhang}
}

\newtheorem{theorem}{Theorem}[section]
\newtheorem{proposition}[theorem]{Proposition}
\newtheorem{corollary}[theorem]{Corollary}

\theoremstyle{definition}
\newtheorem{definition}[theorem]{Definition}
\newtheorem{example}[theorem]{Example}
\theoremstyle{remark}
\newtheorem{remark}[theorem]{Remark}
\crefname{theorem}{theorem}{theorems}
\Crefname{theorem}{Theorem}{Theorems}
\crefname{proposition}{proposition}{propositions}
\Crefname{proposition}{Proposition}{Propositions}
\crefname{corollary}{corollary}{corollaries}
\Crefname{corollary}{Corollary}{Corollaries}
\crefname{lemma}{lemma}{lemmas}
\Crefname{lemma}{Lemma}{Lemmas}
\crefname{definition}{definition}{definitions}
\Crefname{definition}{Definition}{Definitions}
\crefname{example}{example}{examples}
\Crefname{example}{Example}{Examples}
\crefname{remark}{remark}{remarks}
\Crefname{remark}{Remark}{Remarks}
\crefname{section}{section}{sections}
\Crefname{section}{Section}{Sections}

\newcommand{\kk}{\Bbbk}
\newcommand{\Z}{\mathbb Z}
\newcommand{\N}{\mathbb N}
\newcommand{\R}{\mathbb R}
\newcommand{\Mpol}{\mathcal M}
\newcommand{\Ppol}{\mathcal P}
\newcommand{\Bpol}{\mathcal B}
\newcommand{\Spol}{\mathcal S}
\newcommand{\Fpol}{\mathcal F}
\newcommand{\Dpol}{\Delta}
\newcommand{\rank}{\operatorname{rank}}
\newcommand{\im}{\operatorname{im}}
\newcommand{\id}{\operatorname{id}}
\newcommand{\TScript}{\mathbf{TScript}}
\newcommand{\Ch}{\mathbf{Ch}}
\newcommand{\vol}{\operatorname{vol}}
\newcommand{\BM}{\mathrm{BM}}

\title{\textbf{A Chain- and Diagram-Level Semantics for Morphological Calculus}\\[0.4em]
\large Refinement, monodromy, and bivector orbit decompositions}
\author{%
\normalsize Baruch Schneider${}^{1,*}$,\quad
Diana Schneiderov\'a (Barseghyan)${}^{1}$,\quad
Yifan Zhang${}^{1,2,3}$\\[0.6em]
\footnotesize ${}^{1}$Department of Mathematics, University of Ostrava, 702 00 Ostrava, Czechia\\
\footnotesize ${}^{2}$Department of Algebra, Charles University, 186 75 Prague, Czechia\\
\footnotesize ${}^{3}$Department of Applied Mathematics, VSB--Technical University of Ostrava, 708 00 Ostrava, Czechia\\[0.35em]
\footnotesize ${}^{*}$Corresponding author\\
\footnotesize E-mail addresses: \href{mailto:baruch.schneider@osu.cz}{baruch.schneider@osu.cz},
\href{mailto:diana.schneiderova@osu.cz}{diana.schneiderova@osu.cz},
\href{mailto:yifan.zhang@osu.cz}{yifan.zhang@osu.cz}
}
\date{}

\begin{document}
\maketitle

\begin{abstract}
Morphological calculus represents decompositions of geometric objects by polynomial-like expressions in a symbol for the real line, but scalar operations suppress the incidence and attachment maps needed to reconstruct the space. We formulate a finite semantics using script chain complexes. The cell-count polynomial satisfies

$$
\mathcal M_S(t)=\mathcal P_{S,\mathbb K}(t)+(1+t)\mathcal B_{S,\mathbb K}(t),
$$

where $\mathcal P$ is the Poincar'e polynomial and $\mathcal B$ records boundary ranks. Over $\mathbb Z$, Smith labels separate unit pairs, representing refinement overhead, from non-unit pairs carrying torsion. We distinguish Cartesian products from bundles, derive the monodromy defect for mapping tori, and replace scalar gluing by a finite bar construction. Mapping cones, joins, and double mapping cylinders arise as reduced models.

We apply the construction to bivector discrepancies in Sommen's calculations. In dimension four, Hodge decomposition identifies the unit bivector sphere with $S^2*S^2$; the angular excess is a contractible relative complex from inserting the rank-two midpoint. In dimension five, the $SO(5)$-action on $S(\Lambda^2\mathbb R^5)=S^9$ has principal orbit $SO(5)/T^2$ and singular orbits $\widetilde G_2(\mathbb R^5)$ and $\mathbb{CP}^3$. For standard Bruhat cell structures, the reduced double-mapping-cylinder inventory differs from sphere homology by seven unit-labelled pairs. We also treat Borel--Moore realizations of selected noncompact symbols, define division as an action-certified partial operation, and describe the orbit-type face diagram for unit bivectors in dimension six. There the correction polynomial has nonnegative coefficients, while compatibility of the face attachments remains open. These examples separate subdivision and attachment data from quotient actions and support conditions suppressed by scalar notation.
\end{abstract}


\section{Introduction}
\label{sec:intro}

Sommen's morphological calculus treats geometric objects as generalized natural numbers. Disjoint decompositions are written additively, Cartesian products or fibre bundles multiplicatively, and quotient constructions by division. This language produces compact expressions for spheres, projective spaces, Lie groups, Grassmannians, null cones, and twistor spaces \cite{Sommen1999,Sommen2016}. The same papers stress that ``morphological quantity'' has no independent mathematical definition: one object may receive several expressions, different objects may share an expression, and some stratified calculations fail to reproduce the expected ambient quantity.

Script geometry provides a finite setting in which these failures can be analyzed. A script is a chain complex of free abelian groups generated by cells, with boundary operator \(\partial^2=0\) \cite{CerejeirasEtAl2016,PrimerScript2019,CerejeirasKahlerLegatiuk2025}. The cell counts retain the scalar inventory of morphological calculus, whereas the differential retains incidence and attachment information. The resulting loss of information can therefore be measured at chain level. The basic identity
\[
 \sum_k c_k t^k=\sum_k\beta_k t^k+(1+t)Q(t),\qquad Q(t)\in\mathbb N[t],
\]
has the form of the strong polynomial Morse inequalities. For script complexes, the coefficient of \(t^k\) in \(Q\) is exactly the rank of \(\partial_{k+1}\). Integral Smith normal form then separates unit pairs, which are removable algebraic refinement overhead, from non-unit pairs, which encode torsion.

Ambiguous multiplication and gluing are replaced by constructions that retain the relevant maps. Cartesian products are modeled by tensor products of chain complexes. A fibre bundle requires monodromy and spectral-sequence data; for mapping tori, the mapping-cone exact sequence gives an exact correction to the naive product polynomial. More generally, a finite direct diagram of script complexes has a bar realization modeling its homotopy colimit. The same construction specializes to mapping cones, joins, and double mapping cylinders.

The main applications concern the unresolved bivector ``bills'' in \cite{Sommen2016}. After radial normalization, the four-dimensional self-dual and anti-self-dual decomposition turns the unit bivector sphere into the join \(S^2*S^2\). Sommen's rank-two case inserts a midpoint into the join coordinate, so its angular excess is a unit-labelled subdivision complex. In five dimensions, the normalized $SO(5)$-action has cohomogeneity one: the rank-two and equal-singular-value loci are the two singular endpoint orbits, and the principal orbit is the full flag manifold \(SO(5)/T^2\). The required finite datum is the span of the two homogeneous projections; its mapping cone has the homotopy type of \(S^9\).

The algebraic ingredients used below are classical. Their role here is to provide a finite semantics for morphological calculus and to resolve specific ambiguities in Sommen's examples. The main results are the chain-level defect formula and its integral Smith refinement, a diagrammatic treatment of gluing that retains the attaching maps, the identification of Sommen's four-dimensional bivector excess with midpoint subdivision of the join $S^2*S^2$, and the five-dimensional cohomogeneity-one reconstruction by a double mapping cylinder. We also consider Borel--Moore models for selected open strata, action-certified quotients, and the first higher-rank example in dimension six. No complete invariant of topological spaces or general constructible theory is asserted.

The construction should also be distinguished from existing additive cut-and-paste frameworks. Classical scissors-congruence and cut-and-paste theories impose relations generated by decomposing and regluing manifolds; modern refinements recover the corresponding $SK$ groups and lift additive invariants such as Euler characteristic to algebraic $K$-theoretic objects \cite{HoekzemaEtAl2022}. In algebraic geometry, Grothendieck groups of varieties similarly encode cut-and-paste relations, and Bittner's presentation describes the universal Euler characteristic with compact support in characteristic zero \cite{Bittner2004}. These theories identify decompositions through scissor relations. The construction used here retains boundary operators, attaching maps, monodromy, and quotient-action data before any passage to an additive class. It is therefore not a new universal cut-and-paste invariant, but a chain- and diagram-level record of data suppressed by a scalar morphological bill.

\section{Scripts and finite morphological quantities}
\label{sec:scripts}

We use only the chain-complex part of script geometry. In the terminology of \cite{CerejeirasEtAl2016,CerejeirasKahlerLegatiuk2025}, a script consists of free modules generated by cells, including an accumulator module in degree $-1$, together with a boundary operator $\partial$ satisfying $\partial^2=0$. Tight scripts and geomaps form a category $\TScript$ \cite{CerejeirasKahlerLegatiuk2025}. The algebraic constructions below apply to every finite script.

\begin{definition}[Reduced script complex]
Let $S$ be a finite script. Discarding the accumulator degree gives
\[
  C(S):\quad 0\longleftarrow C_0(S)\xleftarrow{\partial_1}C_1(S)
  \xleftarrow{\partial_2}\cdots\xleftarrow{\partial_n}C_n(S)\longleftarrow 0,
\]
where $C_k(S)$ is freely generated by the $k$-cells. For a field $\kk$, put $C(S;\kk)=C(S)\otimes_\Z\kk$.
\end{definition}

\begin{definition}[Three polynomials]
For a finite script $S$ and field $\kk$, set
\[
 c_k=\dim_\kk C_k(S;\kk),\qquad
 \beta_k=\dim_\kk H_k(C(S;\kk)),\qquad
 r_k=\rank_\kk\partial_k,
\]
with $r_0=0$. Define
\begin{align*}
 \Mpol_S(t)&=\sum_{k\ge0}c_kt^k,\\
 \Ppol_{S,\kk}(t)&=\sum_{k\ge0}\beta_kt^k,\\
 \Bpol_{S,\kk}(t)&=\sum_{k\ge0}r_{k+1}t^k.
\end{align*}
\end{definition}

We use the same notation for bounded finite free chain complexes. In that setting $\Mpol_C$ uses the integral ranks of the chain groups, while $\Ppol_{C,\kk}$ and $\Bpol_{C,\kk}$ are computed after base change to the indicated field.

The cell-count polynomial is a finite analogue of a morphological expression $a_0R^n+\cdots+a_n$ when coefficients count open cells by dimension.

\begin{proposition}[Addition and Cartesian multiplication]
\label{prop:semiring}
For finite scripts $S,T$ and a field $\kk$,
\begin{align*}
 \Mpol_{S\oplus T}&=\Mpol_S+\Mpol_T,&
 \Ppol_{S\oplus T}&=\Ppol_S+\Ppol_T,\\
 \Mpol_{S\otimes T}&=\Mpol_S\Mpol_T,&
 \Ppol_{S\otimes T}&=\Ppol_S\Ppol_T.
\end{align*}
Here tensor product refers to the total tensor product of reduced chain complexes.
\end{proposition}

\begin{proof}
The cell-count identities follow by counting basis elements. The homological product formula is the K\"unneth theorem over a field.
\end{proof}

Thus direct sum gives a rigorous additive realization, while tensor product gives a rigorous Cartesian-product realization. Arbitrary fibre bundles require extra data; see \Cref{sec:division,sec:mapping}.

\section{The cell-count decomposition}
\label{sec:normal}

\begin{theorem}[Cell-count decomposition over a field]
\label{thm:normal}
For every finite script $S$ and field $\kk$,
\begin{equation}
\label{eq:normal}
  \boxed{\Mpol_S(t)=\Ppol_{S,\kk}(t)+(1+t)\Bpol_{S,\kk}(t).}
\end{equation}
In particular, $\Mpol_S-\Ppol_{S,\kk}$ is divisible by $1+t$ with nonnegative coefficients.
\end{theorem}

\begin{proof}
Rank--nullity and $\partial_k\partial_{k+1}=0$ give
\[
 \beta_k=(c_k-r_k)-r_{k+1},
\]
so $c_k-\beta_k=r_k+r_{k+1}$. Summation gives \eqref{eq:normal}.
\end{proof}

\begin{remark}[Relation with Morse inequalities]
Equation \eqref{eq:normal} has the form of the strong polynomial Morse inequalities. In a Morse or discrete Morse complex, $\Mpol$ counts critical cells and $\Ppol$ records homology; see \cite{Milnor1963,Forman1998,Forman2002}. The script interpretation identifies the correction polynomial explicitly with the ranks of the cellular boundary maps. We use this identity as a semantics for finite morphological quantities rather than as a new algebraic statement.
\end{remark}

\begin{definition}[Elementary contractible pair]
For $k\ge1$, let $E_k$ be
\[
 0\longleftarrow \kk e_{k-1}\xleftarrow{\id}\kk e_k\longleftarrow0,
\]
concentrated in degrees $k-1,k$.
\end{definition}

\begin{theorem}[Splitting over a field]
\label{thm:fieldsplit}
Every finite-dimensional chain complex $C$ over a field admits a non-canonical chain isomorphism
\[
 C\cong H(C)\oplus\bigoplus_{k\ge1}E_k^{\oplus r_k},
\]
where $H(C)$ has zero differential and $r_k=\rank\partial_k$.
\end{theorem}

\begin{proof}
Let $Z_k=\ker\partial_k$ and $B_k=\im\partial_{k+1}$. Choose complements $Z_k=B_k\oplus H'_k$ and $C_k=Z_k\oplus L_k$. The restriction $\partial_k:L_k\to B_{k-1}$ is an isomorphism. The $H'_k$ give homology with zero differential, and each $L_k\to B_{k-1}$ splits into $r_k$ copies of $E_k$.
\end{proof}

\begin{corollary}[Euler evaluation]
\label{cor:euler}
For every finite script,
\[
 \Mpol_S(-1)=\Ppol_{S,\kk}(-1)=\chi(S;\kk).
\]
\end{corollary}

\begin{corollary}[Boundary polynomial of a product]
For finite scripts over a field,
\[
 \Bpol_{S\otimes T}=\Ppol_S\Bpol_T+\Ppol_T\Bpol_S+(1+t)\Bpol_S\Bpol_T.
\]
\end{corollary}

\begin{proof}
Insert $\Mpol=\Ppol+(1+t)\Bpol$ into Proposition~\ref{prop:semiring} and divide by $1+t$.
\end{proof}

\section{Integral Smith labels}
\label{sec:smith}

The field identity cannot distinguish a unit differential from multiplication by an integer $d>1$.

\begin{definition}[Integral two-term pair]
For $k\ge1$ and $d\in\Z_{>0}$, let
\[
 T_k(d):\quad 0\longleftarrow\Z e_{k-1}\xleftarrow{d}\Z e_k\longleftarrow0.
\]
The pair $T_k(1)=E_k$ is contractible, while for $d>1$ it contributes $\Z/d$ to $H_{k-1}$.
\end{definition}

\begin{theorem}[Integral elementary decomposition]
\label{thm:smith}
Let $C$ be a bounded chain complex of finitely generated free abelian groups. Write
\[
 H_j(C;\Z)\cong\Z^{\beta_j}\oplus\bigoplus_{i=1}^{s_j}\Z/d_{j,i},
 \qquad 2\le d_{j,1}\mid\cdots\mid d_{j,s_j},
\]
and let $r_k=\rank_\mathbb Q(\partial_k\otimes\mathbb Q)$. After integral changes of basis,
\[
 C\cong
 \bigoplus_{j\ge0}\Z[j]^{\oplus\beta_j}
 \oplus\bigoplus_{k\ge1}
 \left(
   \bigoplus_{i=1}^{s_{k-1}}T_k(d_{k-1,i})
   \oplus T_k(1)^{\oplus(r_k-s_{k-1})}
 \right).
\]
\end{theorem}

\begin{proof}
Put $Z_k=\ker\partial_k$ and $B_k=\im\partial_{k+1}$. Since $B_{k-1}$ is a free abelian group, the exact sequence
\[
 0\longrightarrow Z_k\longrightarrow C_k\xrightarrow{\partial_k}B_{k-1}\longrightarrow0
\]
splits; choose $C_k=Z_k\oplus L_k$, so that $\partial_k:L_k\to B_{k-1}$ is an isomorphism. Apply Smith normal form to the inclusion $B_k\hookrightarrow Z_k$. Thus one can choose a basis $z_1,\ldots,z_m$ of $Z_k$ for which $B_k$ is generated by
\[
 d_{k,1}z_1,\ldots,d_{k,s_k}z_{s_k},z_{s_k+1},\ldots,z_{r_{k+1}},
\]
where $2\le d_{k,1}\mid\cdots\mid d_{k,s_k}$ and the omitted basis vectors represent the free part of $H_k(C;\Z)$. Transport this basis of $B_k$ through the isomorphism $L_{k+1}\to B_k$. The resulting pairs are $T_{k+1}(d_{k,i})$ for the torsion factors and $T_{k+1}(1)$ for the remaining boundary generators; the unpaired basis vectors give $\Z[k]$. Performing this simultaneously in every degree yields the displayed decomposition. This is the standard elementary decomposition of a bounded finite free complex over a PID; compare \cite[Sec.~3.6]{Weibel1994}.
\end{proof}

\begin{definition}[Smith-labelled boundary profile]
Let $\mathcal A$ be the free commutative monoid on symbols $[d]$, $d\in\Z_{>0}$. Define
\[
 \Spol_C(t)=\sum_{k\ge1}
 \left((r_k-s_{k-1})[1]+\sum_{i=1}^{s_{k-1}}[d_{k-1,i}]\right)t^{k-1}.
\]
Forgetting labels, $[d]\mapsto1$, gives $|\Spol_C|=\Bpol_{C,\mathbb Q}$.
\end{definition}

A label $[1]$ is algebraically removable. In the geometric subdivision classes treated in \Cref{sec:refinement}, such labels record refinement overhead. A label $[d]$, $d>1$, records torsion.

\begin{example}[Torus, Klein bottle, and projective plane]
Up to integral chain isomorphism,
\begin{align*}
 C(T^2)&\cong\Z[0]\oplus\Z[1]^{\oplus2}\oplus\Z[2],\\
 C(K)&\cong\Z[0]\oplus\Z[1]\oplus T_2(2),\\
 C(\mathbb{RP}^2)&\cong\Z[0]\oplus T_2(2).
\end{align*}
Thus $T^2$ and $K$ share the cell polynomial $(1+t)^2$ but have Smith profiles $0$ and $[2]t$, respectively.
\end{example}

\begin{remark}
The basis changes in \Cref{thm:smith} need not be geomaps and may mix geometric cells. The theorem classifies the underlying free chain complex, not the script as a geometric object.
\end{remark}

\section{Elementary expansions and Euler collapse}
\label{sec:collapse}

\begin{definition}[Algebraic elementary expansion]
A complex $C'$ is an elementary expansion of $C$ in degree $k$ if $C'\cong C\oplus E_k$. An elementary collapse is the reverse operation.
\end{definition}

\begin{proposition}[Expansion overhead]
If $C'=C\oplus E_k$, then
\[
 \Mpol_{C'}=\Mpol_C+t^{k-1}+t^k,\qquad
 \Ppol_{C'}=\Ppol_C,\qquad
 \Bpol_{C'}=\Bpol_C+t^{k-1}.
\]
\end{proposition}

\begin{theorem}[Universal cell-count collapse]
\label{thm:collapse}
Let $I\subset\Z[t]$ be the ideal generated by $t^{k-1}+t^k$ for $k\ge1$. Then
\[
 \Z[t]/I\cong\Z[t]/(1+t)\cong\Z,
\]
where the final map is evaluation at $t=-1$. More precisely, let $A$ be an abelian group and let $F:\N[t]\to A$ be an additive monoid homomorphism. If
\[
 F(p+t^{k-1}+t^k)=F(p)
\]
for every $p\in\N[t]$ and $k\ge1$, then there is a unique group homomorphism $\varphi:\Z\to A$ such that
\[
 F(p)=\varphi\bigl(p(-1)\bigr)
\]
for every $p\in\N[t]$.
\end{theorem}

\begin{proof}
The generators are $t^{k-1}(1+t)$, so $I=(1+t)$. Additivity extends $F$ uniquely to the Grothendieck group $\Z[t]$ of $\N[t]$. Expansion invariance makes the extension vanish on $I$, hence it factors uniquely through $\Z[t]/(1+t)\cong\Z$.
\end{proof}

This theorem concerns only invariants of the cell polynomial. Script geometry avoids the collapse by retaining $\partial$; homology avoids it by passing to $\Ppol$.

\section{Geometric refinement overhead}
\label{sec:refinement}

The categorical treatment of script geometry models refinement by injective geomaps and cell expansions \cite{CerejeirasKahlerLegatiuk2025}. The following moves produce unit-labelled relative complexes.

\subsection{Edges and polygonal faces}

\begin{proposition}[Edge subdivision]
\label{prop:edge}
Replace an oriented edge $e=[v_0,v_1]$ by $e_1=[v_0,w]$ and $e_2=[w,v_1]$, and define $i(e)=e_1+e_2$. Then $i:C_*(X)\to C_*(X')$ is injective and
\[
 C_*(X')/iC_*(X)\cong E_1.
\]
Thus $\Mpol_{X'}-\Mpol_X=1+t$ and integral homology is unchanged.
\end{proposition}

\begin{proof}
The quotient is generated by $[e_1]$ and $[w]$, with $\partial[e_1]=[w]$.
\end{proof}

\begin{proposition}[Binary face splitting]
\label{prop:face}
Let a regular polygonal $2$-cell $f$ be split into $f_1,f_2$ by a new edge $a$ whose endpoints already lie in the old boundary, with orientations chosen so $\partial f_1+\partial f_2=\partial f$. Under $i(f)=f_1+f_2$,
\[
 C_*(X')/iC_*(X)\cong E_2.
\]
Thus $\Mpol_{X'}-\Mpol_X=t+t^2$.
\end{proposition}

\begin{proof}
The relative generators are $[f_1]$ and $[a]$, and the relative boundary is multiplication by $\pm1$.
\end{proof}

\begin{corollary}[Low-dimensional factorization]
If $X'$ is obtained from a graph or polygonal $2$-complex $X$ by $a$ edge subdivisions and $b$ binary face splits, then the successive relative quotients are $E_1^{\oplus a}$ and $E_2^{\oplus b}$ and
\[
 \Mpol_{X'}-\Mpol_X=(1+t)(a+bt).
\]
All Smith labels are units.
\end{corollary}

\subsection{A general unit-defect criterion}

The preceding refinements are geometric examples of a general algebraic fact. The definition of refinement in current script geometry requires an injective geomap, but injectivity alone does not imply preservation of homology \cite{CerejeirasKahlerLegatiuk2025}. The following criterion separates the part that can be proved formally from the additional geometric hypothesis.

\begin{theorem}[Unit defect of a homology-preserving refinement]
\label{thm:generalrefinement}
Let $g:C\to C'$ be a quasi-isomorphism between bounded chain complexes of finitely generated free abelian groups.
\begin{enumerate}[label=\textup{(\alph*)}]
 \item The mapping cone $\operatorname{Cone}(g)$ is a finite free contractible complex. Consequently every nonzero Smith factor in its elementary decomposition is a unit.
 \item If, in addition, $g$ is degreewise split injective, then the quotient complex $Q=C'/gC$ is finite free and contractible. Hence
 \[
   Q\cong\bigoplus_{k\ge1}E_k^{\oplus m_k}
 \]
 for uniquely determined ranks $m_k\ge0$, and
 \[
   \Mpol_{C'}(t)-\Mpol_C(t)
   =(1+t)\sum_{k\ge1}m_k t^{k-1}.
 \]
\end{enumerate}
\end{theorem}

\begin{proof}
A quasi-isomorphism has acyclic mapping cone. A bounded acyclic complex of finitely generated free abelian groups is split exact: its cycle and boundary groups are free, and the short exact sequences defining them split. Thus the cone is contractible and its Smith decomposition consists only of $T_k(1)=E_k$ blocks. If $g$ is degreewise split injective, then $Q$ is degreewise free. The long exact homology sequence of
\[
 0\longrightarrow C\xrightarrow{g}C'\longrightarrow Q\longrightarrow0
\]
shows that $Q$ is acyclic, and the same splitting argument gives the stated decomposition and polynomial identity.
\end{proof}

\begin{corollary}[What remains geometric]
Any refinement geomap that is both a quasi-isomorphism and degreewise split injective has only unit-labelled refinement overhead. Therefore the unresolved issue for general script refinement is not Smith theory but the geometric question of when an injective refinement geomap satisfies these two additional properties.
\end{corollary}

\begin{remark}
The edge and polygonal-face refinements above are primitive split injections, and the join refinement below has the same property; thus \Cref{thm:generalrefinement} recovers their unit-only conclusion. The theorem does not assert that every refinement in $\TScript$ is homology preserving; the published definition supplies injectivity and boundary compatibility, not this stronger conclusion \cite{CerejeirasKahlerLegatiuk2025}.
\end{remark}

\subsection{Subdivision of a join coordinate}

For finite CW complexes $X,Y$, write their join as
\[
 X*Y=(X\times Y\times[0,1])/\sim,
\]
where the $Y$ coordinate is forgotten at $0$ and the $X$ coordinate at $1$.

\begin{theorem}[Midpoint subdivision of a join]
\label{thm:joinref}
Equip $X*Y$ with the standard product-join CW structure: the endpoint copies contribute the cells of $X$ and $Y$, and for every pair of open cells $\sigma\subset X$, $\tau\subset Y$ the image of $\sigma\times\tau\times(0,1)$ is an open join cell $\sigma*\tau$ of dimension $\dim\sigma+\dim\tau+1$. Let $(X*Y)'$ be the common subdivision obtained by inserting the midpoint $1/2$ in every such join cell, and let $s:C_*(X*Y)\to C_*((X*Y)')$ be the standard cellular subdivision chain map. Then the subdivision map can be chosen injective and
\[
 C_*((X*Y)')/sC_*(X*Y)\cong C_*(X\times Y)\otimes E_1.
\]
Consequently,
\begin{align*}
 \Mpol_{(X*Y)'}-\Mpol_{X*Y}&=(1+t)\Mpol_X\Mpol_Y,\\
 H_*\bigl(C_*((X*Y)')/sC_*(X*Y)\bigr)&=0,
\end{align*}
and every Smith label in the algebraic elementary decomposition of the relative complex is $[1]$.
\end{theorem}

\begin{proof}
For each product cell $\sigma\times\tau$ in $X\times Y$, midpoint subdivision replaces the corresponding open join cell by two cells of dimension $\dim\sigma+\dim\tau+1$ and inserts one midpoint cell of dimension $\dim\sigma+\dim\tau$. Relative to the unsubdivided structure, these generators form the total tensor product with the interval subdivision pair $E_1$. Since $E_1$ is contractible, so is the tensor product. Its cell polynomial is $(1+t)\Mpol_X\Mpol_Y$. A finite free contractible complex has only unit factors in its Smith decomposition.
\end{proof}

\begin{remark}
The theorem uses the explicit product subdivision of the join cylinder. It is not a claim that the old cells literally form a subcomplex of the new CW structure; the injection is the standard cellular subdivision chain map. The stated relative complex and polynomial overhead refer to this chosen product subdivision.
\end{remark}

\section{The same quantity can encode different spaces}
\label{sec:examples}

\begin{example}[Torus and Klein bottle]
Both $T^2$ and $K$ have a CW structure with one $0$-cell, two $1$-cells, and one $2$-cell, so
\[
 \Mpol_{T^2}=\Mpol_K=(1+t)^2.
\]
Over $\mathbb Q$,
\[
 \Ppol_{T^2}=1+2t+t^2,\qquad \Bpol_{T^2}=0,
\]
whereas the Klein-bottle attaching map has exponent-sum vector $(0,2)$ and
\[
 \Ppol_K=1+t,\qquad \Bpol_K=t,\qquad \Spol_K=[2]t.
\]
Thus
\[
 1+2t+t^2=(1+t)+(1+t)t.
\]
The unlabelled correction records rational cancellation; the label $2$ records $\Z/2$ torsion.
\end{example}

\begin{example}[Dependence on coefficients]
Over $\mathbb F_2$, the coefficient $2$ in the Klein-bottle boundary vanishes, so $\Ppol_{K,\mathbb F_2}=1+2t+t^2$. The integral label retains information lost by a field choice.
\end{example}

\section{Products and bundles}
\label{sec:products}

\begin{proposition}[Tensor-product obstruction]
\label{prop:obstruction}
If
\[
 C(S;\kk)\simeq C(A;\kk)\otimes C(B;\kk)
\]
by chain homotopy, then
\[
 \Ppol_{S,\kk}=\Ppol_{A,\kk}\Ppol_{B,\kk}.
\]
Therefore failure of Poincar\'e-polynomial factorization obstructs a chain-level realization of a proposed morphological product.
\end{proposition}

In particular, the Klein bottle is not a rational chain product of two circles although its cell polynomial is $(1+t)^2$.

\section{Division as an action-certified partial operation}
\label{sec:division}

Sommen's notation uses division for projective spaces, homogeneous spaces, and the reversal of product-like formulas. These uses cannot be modeled by inversion in the cell-count semiring. The missing datum is usually an action.

\begin{definition}[Action-certified division]
Let $G$ be a topological group acting continuously and properly on a space $X$ through an action $\rho:G\curvearrowright X$. Define
\[
  \operatorname{Div}(X;G,\rho):=X/G.
\]
If the action is free, we call this a \emph{principal division certificate}. For a closed subgroup $H\subset G$, homogeneous division
\[
  G{/}H
\]
means the quotient for the specified right action of $H$ on $G$, so the inclusion $H\hookrightarrow G$ is part of the data.
\end{definition}

The notation $X/G$ is therefore shorthand for a triple $(X,G,\rho)$, not a binary algebraic operation on quantities. If the action is not free, the ordinary quotient $X/G$ and the homotopy quotient $EG\times_GX$ are different constructions and should not be identified.

\begin{proposition}[Finite cellular quotients]
\label{prop:finitequotient}
Let a finite group $\Gamma$ act freely and cellularly on a finite CW complex $X$, with cells permuted without inversions. Then $X/\Gamma$ is a finite CW complex and
\[
 C_*(X/\Gamma;\Z)\cong C_*(X;\Z)_\Gamma,
\]
where the right-hand side denotes coinvariants of the cellular chain complex.
\end{proposition}

\begin{proof}
The cells of the quotient are the $\Gamma$-orbits of cells of $X$. Choosing one orientation in each orbit identifies the free chain group of the quotient with the coinvariants, and the cellular boundary descends because the action is cellular. This is the free-action specialization of the cellular quotient framework used for finite $G$-CW complexes; compare \cite{Knudson2004}.
\end{proof}

\begin{example}[Projective divisions]
The familiar formulas
\[
 \mathbb{RP}^n=S^n/(\mathbb Z/2),\qquad
 \mathbb{CP}^n=S^{2n+1}/S^1,\qquad
 \mathbb{HP}^n=S^{4n+3}/Sp(1)
\]
are valid divisions because the antipodal and Hopf actions are specified and free. This is the additional datum suppressed when one writes only a numerator and denominator.
\end{example}

\begin{example}[The same numerator and denominator need not determine the quotient]
Different free actions of the same finite cyclic group on an odd sphere produce different lens spaces. Thus a formal expression such as $S^{2m-1}/(\mathbb Z/p)$ does not determine a space until the action is supplied.
\end{example}

\begin{proposition}[When scalar division happens to work]
\label{prop:divisioncriterion}
Let $F\to E\to B$ be a fibration of finite $\kk$-homology type. If the monodromy on $H^*(F;\kk)$ is trivial and the Serre spectral sequence collapses at $E_2$, then
\[
 \Ppol_{E,\kk}(t)=\Ppol_{B,\kk}(t)\Ppol_{F,\kk}(t).
\]
Only under such additional hypotheses can one recover the Poincar\'e polynomial of the base by polynomial division.
\end{proposition}

\begin{proof}
Trivial monodromy gives $E_2^{p,q}=H^p(B;\kk)\otimes H^q(F;\kk)$. If the spectral sequence collapses, the total graded dimensions are therefore the product of the graded dimensions of base and fibre.
\end{proof}

The hypotheses are sufficient, not necessary. Their role is to emphasize that a quotient or fibre-bundle certificate is geometric data, while scalar factorization is a secondary consequence that may or may not occur.

\begin{remark}[A quotient for which polynomial division fails]
For the Hopf fibration $S^1\to S^3\to S^2$,
\[
 \Ppol_{S^3}(t)=1+t^3,
 \qquad
 \Ppol_{S^1}(t)=1+t,
\]
and $(1+t^3)/(1+t)$ is not $1+t^2$. The geometric quotient exists, but the Serre differential carries information omitted by scalar division. Accordingly, division is defined here first by quotient data and only, when a collapse theorem permits it, by division of scalar invariants.
\end{remark}

\begin{remark}[Phantom divisions]
This viewpoint gives a precise status to Sommen's ``phantom'' quotients. An expression with no supplied group action, homogeneous inclusion, or other quotient certificate remains a formal morphological expression. It may admit a geometric synthesis, but it is not assigned an ordinary quotient merely because an algebraic denominator is present.
\end{remark}

\section{Mapping tori and the exact monodromy defect}
\label{sec:mapping}

Let $f:X\to X$ be a cellular self-map of a finite CW complex and
\[
 T_f=X\times[0,1]/(x,1)\sim(f(x),0)
\]
its mapping torus.

\begin{definition}[Fixed and moved homology]
For a field $\kk$, define
\begin{align*}
 \nu_j(f;\kk)&=\dim_\kk\ker(1-f_*:H_j(X;\kk)\to H_j(X;\kk)),\\
 \Fpol_{f,\kk}(t)&=\sum_j\nu_j(f;\kk)t^j,\\
 \Dpol_{f,\kk}(t)&=\Ppol_{X,\kk}(t)-\Fpol_{f,\kk}(t)
 =\sum_j\rank(1-f_*|H_j)t^j.
\end{align*}
\end{definition}

\begin{theorem}[Mapping-torus formula]
\label{thm:mapping}
For every field $\kk$,
\[
 \boxed{\Ppol_{T_f,\kk}(t)=(1+t)\Fpol_{f,\kk}(t).}
\]
Equivalently,
\[
 (1+t)\Ppol_{X,\kk}-\Ppol_{T_f,\kk}=(1+t)\Dpol_{f,\kk}.
\]
\end{theorem}

\begin{proof}
Because $f$ is cellular, the standard cellular mapping-torus model has chain complex chain-homotopy equivalent to the algebraic mapping cone
\[
 \operatorname{Cone}\bigl(1-f_\#:C_*^{\mathrm{cell}}(X;\kk)\longrightarrow C_*^{\mathrm{cell}}(X;\kk)\bigr).
\]
Equivalently, the mapping torus is the homotopy coequalizer of $\id_X$ and $f$, and cellular chains give the corresponding cone model. The long exact sequence of a mapping cone \cite{Weibel1994,Hatcher2002} therefore yields short exact sequences of vector spaces
\[
 0\to\operatorname{coker}(1-f_*|H_j(X;\kk))\to H_j(T_f;\kk)
 \to\ker(1-f_*|H_{j-1}(X;\kk))\to0.
\]
For an endomorphism of a finite-dimensional vector space, kernel and cokernel have equal dimension. Hence $\beta_j(T_f)=\nu_j+\nu_{j-1}$. When $f$ is a homeomorphism this exact sequence is the usual Wang sequence for the bundle over $S^1$; the cone proof shows that the displayed formula remains valid for every cellular self-map.
\end{proof}

\begin{corollary}[Exact product criterion over the circle]
\[
 \Ppol_{T_f,\kk}=(1+t)\Ppol_{X,\kk}
\]
if and only if $f_*=\id$ on $H_*(X;\kk)$.
\end{corollary}

\begin{example}[Torus and Klein bottle]
For $X=S^1$, identity monodromy gives $T_f=T^2$ and $(1+t)^2$. Reflection acts by $-1$ on $H_1(S^1;\kk)$; if $\operatorname{char}\kk\ne2$, then $\Fpol_f=1$, $\Dpol_f=t$, and $\Ppol_{T_f}=1+t$.
\end{example}

\begin{corollary}[Torus mapping tori]
Let $A\in GL(m,\Z)$ induce $f_A:T^m\to T^m$. Then
\[
 \Ppol_{T_{f_A},\kk}(t)=(1+t)\sum_{j=0}^m
 \dim_\kk\ker(I-\Lambda^jA)\,t^j.
\]
\end{corollary}

\begin{example}[Inversion]
For $A=-I_m$ and $\operatorname{char}\kk\ne2$, only even exterior powers are fixed, so
\[
 \Ppol_{T_{f_A}}(t)=\frac{1+t}{2}\big((1+t)^m+(1-t)^m\big).
\]
\end{example}

\section{Finite diagrammatic semantics via homotopy colimits}
\label{sec:hocolim}

Ordinary addition and multiplication do not retain the maps in a gluing diagram. We therefore use diagrams of finite chain complexes rather than scalar expressions. We work in the category $\Ch^b_{\mathrm{fr}}(\Z)$ of bounded chain complexes of finitely generated free abelian groups.

\begin{definition}[Diagrammatic script]
A finite diagrammatic script is a functor
\[
 D:\mathcal I\longrightarrow \Ch^b_{\mathrm{fr}}(\Z),
\]
where $\mathcal I$ is a finite direct category (no nonidentity endomorphisms or oriented cycles); in particular, one may take a finite poset. The objects $D(i)$ are script complexes and the arrows are chain maps encoding incidence, projection, attachment, or refinement data.
\end{definition}

\begin{definition}[Bar realization]
The simplicial replacement of $D$ is the simplicial chain complex
\[
 B_q(*,\mathcal I,D)=
 \bigoplus_{i_0\to i_1\to\cdots\to i_q}D(i_0),
\]
with the standard bar faces and degeneracies. Let $N B_\bullet(*,\mathcal I,D)$ be its normalized complex in the simplicial direction and define
\[
 \mathfrak C(D)=\operatorname{Tot}N B_\bullet(*,\mathcal I,D).
\]
We call $\mathfrak C(D)$ the derived script realization of the diagram.
\end{definition}

The same normalized bar totalization will also be used, without the finiteness-of-generators restriction, for diagrams of ordinary chain complexes such as singular chains. In the finite indexing categories considered here the simplicial direction is finite, so this comparison introduces no convergence issue.

Define the raw bar inventory by
\[
 \Mpol_{\mathrm{bar}}(D;t):=\Mpol_{\mathfrak C(D)}(t).
\]
For a finite poset, nondegenerate simplices are strict chains, so
\begin{equation}
\label{eq:barbill}
 \Mpol_{\mathrm{bar}}(D;t)
 =\sum_{q\ge0}t^q
   \sum_{i_0<i_1<\ldots<i_q}\Mpol_{D(i_0)}(t).
\end{equation}
This polynomial depends on the chosen bar model. The chain-homotopy or quasi-isomorphism type of $\mathfrak C(D)$ is the invariant object.

\begin{theorem}[Finite derived-diagram realization]
\label{thm:barrealization}
Let $D:\mathcal I\to\Ch^b_{\mathrm{fr}}(\Z)$ be a finite diagrammatic script.
\begin{enumerate}[label=\textup{(\roman*)}]
 \item The bar complex $\mathfrak C(D)$ models the homotopy colimit of $D$ in chain complexes.
 \item An objectwise quasi-isomorphism of diagrams induces a quasi-isomorphism of their derived script realizations.
 \item Let $X:\mathcal I\to\mathbf{CW}$ be a finite diagram of CW complexes and cellular maps, and let
 \[
   S(i)=C_*^{\mathrm{sing}}(X_i;\Z).
 \]
 Then the normalized bar totalization of the singular-chain diagram has a natural quasi-isomorphism
 \[
   \operatorname{Tot}N B_\bullet(*,\mathcal I,S)
   \simeq C_*^{\mathrm{sing}}\bigl(\operatorname{hocolim}_{\mathcal I}X;\Z\bigr),
 \]
 where the topological homotopy colimit is represented by the Bousfield--Kan simplicial replacement. Consequently, if a finite script diagram $D$ is equipped with a natural transformation
 \[
   \eta:D\Longrightarrow S
 \]
 whose components are quasi-isomorphisms, then
 \[
   \mathfrak C(D)\simeq
   C_*^{\mathrm{sing}}\bigl(\operatorname{hocolim}_{\mathcal I}X;\Z\bigr).
 \]
 In particular, cellular or script chain models may be used whenever such a compatible objectwise comparison is supplied.
\end{enumerate}
Consequently,
\begin{equation}
\label{eq:barmorse}
 \Mpol_{\mathrm{bar}}(D;t)
 =\Ppol_{\mathfrak C(D),\kk}(t)
 +(1+t)\Bpol_{\mathfrak C(D),\kk}(t)
\end{equation}
for every field $\kk$.
\end{theorem}

\begin{proof}
The Bousfield--Kan bar construction computes homotopy colimits of chain-complex diagrams; see \cite{Arakawa2026,DuggerHocolim} for this bar model, proving \textup{(i)}. Since the indexing category is finite and the script complexes are bounded, filtering the normalized bar totalization by simplicial degree shows that an objectwise quasi-isomorphism induces a quasi-isomorphism of total complexes, proving \textup{(ii)}. For \textup{(iii)}, apply singular chains degreewise to the topological simplicial replacement. The Eilenberg--Zilber comparison identifies the normalized singular chains of its realization, up to quasi-isomorphism, with the normalized bar totalization of the singular-chain diagram. If a compatible objectwise quasi-isomorphism $\eta:D\Rightarrow S$ is supplied, \textup{(ii)} gives the second comparison. Equation \eqref{eq:barmorse} is \Cref{thm:normal} applied to $\mathfrak C(D)$.
\end{proof}

\begin{remark}[Why compatibility is part of the data]
Part~\textup{(iii)} is stated first for singular chains. Although cellular and singular homology agree objectwise for CW complexes, an arbitrary CW diagram need not carry a preferred strict natural chain map from the chosen cellular complexes to singular chains. A cellular or script model is therefore used only after a compatible objectwise quasi-isomorphism has been specified. In particular, the maps in the diagram are part of the input and cannot be recovered from the scalar inventories.
\end{remark}

\begin{corollary}[Diagrammatic Euler formula]
For a finite poset diagram,
\[
 \chi\bigl(\mathfrak C(D)\bigr)
 =\sum_{q\ge0}(-1)^q
   \sum_{i_0<i_1<\ldots<i_q}\chi\bigl(D(i_0)\bigr).
\]
\end{corollary}

\begin{proof}
Evaluate \eqref{eq:barbill} and \eqref{eq:barmorse} at $t=-1$.
\end{proof}

\begin{proposition}[Reduction of a span]
\label{prop:spanreduction}
For a span of chain maps
\[
 C_*(P)\xrightarrow{p_{-\#}}C_*(B_-),
 \qquad
 C_*(P)\xrightarrow{p_{+\#}}C_*(B_+),
\]
the bar realization admits an integral chain decomposition
\[
 \mathfrak C(D)\cong
 \operatorname{Cone}(p_{-\#},-p_{+\#})
 \oplus\bigl(C_*(P)\otimes E_1\bigr).
\]
Hence the full bar bill exceeds the reduced mapping-cone bill by
\[
 (1+t)\Mpol_P(t).
\]
All of this additional inventory consists of unit-labelled pairs.
\end{proposition}

\begin{proof}
Write $A=C(P)$, $X=C(B_-)$, $Y=C(B_+)$, $f=p_{-\#}$, and $g=p_{+\#}$. With one standard totalization convention, an element of the normalized bar complex is written
\[
 (a,x,y,u,v)\in A_n\oplus X_n\oplus Y_n\oplus A_{n-1}\oplus A_{n-1},
\]
and the differential is
\[
 d(a,x,y,u,v)=
 (d_A(a)-u-v,\ d_X(x)+f(u),\ d_Y(y)+g(v),\ -d_A(u),\ -d_A(v)).
\]
Put
\[
 w=u+v,\qquad z=v,\qquad x'=x+f(a).
\]
This is an integral triangular change of basis. In the new coordinates,
\[
 d(a,x',y,w,z)=
 (d_A(a)-w,\ d_X(x')-f(z),\ d_Y(y)+g(z),\ -d_A(w),\ -d_A(z)).
\]
The variables $(a,w)$ form the contractible summand $A\otimes E_1$, while $(x',y,z)$ form the mapping cone of $(-f,g)$, which is isomorphic by a sign change to the cone of $(f,-g)$. The inventory statement follows by counting generators.
\end{proof}

\begin{remark}[Scope of the diagrammatic realization]
The theorem assigns to every finite direct diagram of script complexes a derived realization, well defined up to quasi-isomorphism. It does not make the resulting chain complex a complete invariant of the underlying space; \Cref{sec:limits} records this limitation.
\end{remark}

\section{Double mapping cylinders and cohomogeneity-one bills}
\label{sec:dmc}

Let $p_-:P\to B_-$ and $p_+:P\to B_+$ be cellular maps of finite CW complexes. Their double mapping cylinder is
\[
 \operatorname{DMC}(p_-,p_+)
 =B_-\cup_{p_-}\bigl(P\times[0,1]\bigr)\cup_{p_+}B_+.
\]
The two maps are part of the structure. Replacing them by the three quantities of $B_-$, $P$, and $B_+$ loses the attaching information. By \Cref{prop:spanreduction}, the mapping-cone model below is the reduced bar realization of this span.

\begin{theorem}[Mapping-cone bill]
\label{thm:dmc}
There is a cellular chain model
\[
 C_*\bigl(\operatorname{DMC}(p_-,p_+)\bigr)
 \cong
 \operatorname{Cone}\!\left(
   (p_{-\#},-p_{+\#}):C_*(P)\longrightarrow C_*(B_-)\oplus C_*(B_+)
 \right).
\]
Consequently its raw cell inventory is
\begin{equation}
\label{eq:dmcbill}
 \Mpol_{\mathrm{DMC}}(t)=\Mpol_{B_-}(t)+\Mpol_{B_+}(t)+t\Mpol_P(t),
\end{equation}
but its homology depends on the two chain maps, not only on the three polynomials.
\end{theorem}

\begin{proof}
Give $P\times[0,1]$ the product CW structure. Besides the endpoint cells, each $k$-cell of $P$ contributes one open prism cell of dimension $k+1$. The cellular differential of a prism has the internal term $-\partial_P$ and the two endpoint terms $p_{-\#}$ and $-p_{+\#}$. This is precisely the mapping-cone differential. Counting generators gives \eqref{eq:dmcbill}.
\end{proof}

\begin{corollary}[Cohomogeneity-one semantics]
\label{cor:cohom1}
Let a compact connected Lie group $G$ act with cohomogeneity one on a compact manifold $M$, with orbit interval and group diagram
\[
 H\subset\{K_-,K_+\}\subset G.
\]
Then
\[
 P=G/H,\qquad B_\pm=G/K_\pm,
\]
and $M$ is the double mapping cylinder of the homogeneous projections $p_\pm:P\to B_\pm$. Equivalently,
\[
 M\cong G\times_{K_-}D^{\ell_-+1}\ \cup_{G/H}\ G\times_{K_+}D^{\ell_++1},
 \qquad K_\pm/H\cong S^{\ell_\pm}.
\]
The relevant datum is therefore the diagram $B_-\xleftarrow{p_-}P\xrightarrow{p_+}B_+$, not the scalar sum $B_-+tP+B_+$ alone.
\end{corollary}

\begin{proof}
This is the standard group-diagram description of a compact cohomogeneity-one manifold as the union of the normal disk bundles of its two non-principal orbits; see, for example, \cite{GroveZiller2000}. Applying \Cref{thm:dmc} gives the chain model.
\end{proof}

\begin{remark}[A Thom-adapted inventory]
If the positive normal disk bundle has rank $d_+$ and is oriented over the chosen coefficients, one may instead start with the negative disk bundle, which retracts to $B_-$, and attach the positive disk bundle relative to its sphere boundary. A compatible Thom CW structure has inventory
\[
 \Mpol_{B_-}(t)+t^{d_+}\Mpol_{B_+}(t).
\]
This inventory is usually smaller than \eqref{eq:dmcbill}; the difference consists of adjacent-dimensional pairs. The attaching homomorphism is still essential.
\end{remark}

\section{Joins as a corrected morphological operation}
\label{sec:joins}

Several of Sommen's sphere formulas are naturally expressed in terms of joins. The join is the double mapping cylinder of the projections $X\times Y\to X$ and $X\times Y\to Y$. Hence its orbit-style inventory is $\Mpol_X+\Mpol_Y+t\Mpol_X\Mpol_Y$, while the projection maps encode the endpoint collapses. Reduced chains satisfy, up to the usual cellular identifications,
\[
 \widetilde C_*(X*Y;\kk)\simeq \Sigma\bigl(\widetilde C_*(X;\kk)\otimes\widetilde C_*(Y;\kk)\bigr).
\]

\begin{proposition}[Homological join formula]
\label{prop:joinP}
For connected finite CW complexes over a field,
\[
 \widetilde\Ppol_{X*Y,\kk}(t)
 =t\,\widetilde\Ppol_{X,\kk}(t)\widetilde\Ppol_{Y,\kk}(t),
\]
or equivalently
\[
 \Ppol_{X*Y,\kk}(t)=1+t\bigl(\Ppol_X(t)-1\bigr)\bigl(\Ppol_Y(t)-1\bigr).
\]
\end{proposition}

\begin{proof}
Use $X*Y\simeq\Sigma(X\wedge Y)$ and the K\"unneth theorem over a field \cite{Hatcher2002}.
\end{proof}

For example, $S^p*S^q\cong S^{p+q+1}$ and the formula gives $1+t^{p+q+1}$ from $1+t^p$ and $1+t^q$. Unlike an ordinary product formula, it incorporates the collapse of one factor at each endpoint of the join parameter.

\section{Selected Borel--Moore models for Sommen's open strata}
\label{sec:borelmoore}

We do not develop a general constructible theory. The noncompact identities considered in this section admit direct interpretations in Borel--Moore homology. For the locally compact Hausdorff, locally contractible spaces used below, let $H_*^{\BM}(X)$ denote homology of locally finite singular chains. With constant coefficients it is computed by the homology of the one-point compactification relative to the added point \cite{BorelMoore1960}. When the groups have finite rank, write
\[
 \Ppol_X^{\BM}(t)=\sum_i \rank H_i^{\BM}(X;\Z)t^i.
\]
Only the examples needed later are considered.

\begin{example}[The line and the relation $R=2R_++1$]
\label{ex:BMline}
Compactify $\R$ to $S^1$ and choose two vertices, $0$ and $\infty$, joined by two oriented edges. Relative to the point at infinity, the cellular complex is
\[
 0\longleftarrow \Z\{0\}
 \xleftarrow{\ (1\ -1)\ }
 \Z\{e_+,e_-\}
 \longleftarrow0.
\]
Its raw inventory is
\[
 1+2t,
\]
which is exactly the stratified bill $1+2R_+$ after assigning one open radial direction degree one. Its Borel--Moore homology is $\Z$ in degree one, so
\[
 1+2t=t+(1+t).
\]
Thus $R=2R_++1$ is not a scalar equality of finite cell counts; it is a refinement identity with one contractible adjacent-dimensional pair.
\end{example}

\begin{proposition}[Punctured Euclidean space]
\label{prop:BMpunctured}
For every $n\ge1$,
\[
 \Ppol_{\R^n\setminus\{0\}}^{\BM}(t)=t^n+t.
\]
This agrees with the polar factorization
\[
 \R^n\setminus\{0\}\cong S^{n-1}\times(0,\infty)
\]
at the Borel--Moore homology level:
\[
 t^n+t=t(1+t^{n-1}).
\]
\end{proposition}

\begin{proof}
The punctured space deformation retracts ordinarily to $S^{n-1}$ and is an oriented $n$-manifold. Borel--Moore Poincar\'e duality identifies $H_i^{\BM}(\R^n\setminus\{0\})$ with $H^{n-i}(\R^n\setminus\{0\})$; hence the only nonzero groups occur in degrees $n$ and $1$ (with the same formula also giving two degree-one generators when $n=1$). The polar product gives the identical polynomial because $(0,\infty)\cong\R$ has Borel--Moore class in degree one.
\end{proof}

\begin{proposition}[Infinite cones and null-cone bills]
\label{prop:BMcone}
Let $L$ be a nonempty finite CW complex and let
\[
 c_\infty L=(L\times[0,\infty))/(L\times\{0\})
\]
be the infinite cone with apex $v$. There is a finite Borel--Moore cellular model with
\[
 D_0=\Z\{v\},\qquad D_{k+1}=C_k(L),
\]
whose differential in degree one is the cellular augmentation $\epsilon:C_0(L)\to\Z$ and in higher degrees is the shifted differential of $L$. Consequently
\[
 H_i(D)\cong \widetilde H_{i-1}(L),\qquad i\ge1,
 \qquad H_0(D)=0.
\]
Equivalently, $(c_\infty L)^+$ is the suspension $\Sigma L$.
\end{proposition}

\begin{proof}
The one-point compactification of the infinite cone adds the second suspension vertex, giving $\Sigma L$. Taking the cellular complex relative to that vertex leaves the apex and one shifted copy of every cell of $L$. The boundary of a shifted $0$-cell hits the apex through the augmentation, while all higher boundaries are shifted boundaries from $L$.
\end{proof}

\begin{example}[Sommen's complex quadratic null cone in real coordinates]
For $n\ge2$, consider Sommen's complex quadratic null cone
\[
 NC_{n-1}=\left\{z=(z_1,\ldots,z_n)\in\mathbb C^n:\sum_{j=1}^n z_j^2=0\right\}.
\]
Writing $z=x+iy$ with $x,y\in\R^n$, the defining equation is equivalent to
\[
 |x|=|y|,\qquad \langle x,y\rangle=0.
\]
For $z\neq0$, radial normalization therefore leaves an orthonormal two-frame, so the nonzero link is $V_{n,2}(\R)$ and topologically
\[
 NC_{n-1}\cong c_\infty V_{n,2}(\R).
\]
For any chosen finite CW structure on $V_{n,2}(\R)$, the raw conical inventory has the form
\[
 1+t\Mpol_{V_{n,2}},
\]
mirroring Sommen's $1+V_{n,2}(\R)R_+$ \cite{Sommen2016}. The additional datum is the augmentation from the radial degree-one cells to the apex. If this map is omitted, the cone attachment is replaced by an ordinary sum. Thus the raw inventory agrees with the conical cell model only when the attaching map is retained. This example should be distinguished from Sommen's later real signature-$(p,q)$ null cones, whose normalized links are products of spheres rather than $V_{n,2}(\R)$.
\end{example}

\begin{remark}
These examples extend the construction to conical and finitely compactifiable strata. They do not provide a closed algebra of arbitrary complements and constructible sets; functorial subtraction and general stratified pushforwards would require a systematic Borel--Moore or Grothendieck framework.
\end{remark}

\section{The bivector sphere as a join}
\label{sec:bivector}

Let $V$ be an oriented Euclidean four-space. The Hodge star gives an orthogonal decomposition
\[
 \Lambda^2V=\Lambda^2_+V\oplus\Lambda^2_-V,
\]
where each summand has dimension three \cite{DonaldsonKronheimer1990}. Write $S_\pm=S(\Lambda^2_\pm V)\cong S^2$.

\begin{theorem}[Hodge--join model of the unit bivector sphere]
\label{thm:bivjoin}
The map
\[
 \Phi:S_+*S_-\longrightarrow S(\Lambda^2V),\qquad
 [u,v,\theta]\longmapsto \cos\theta\,u+\sin\theta\,v,
 \quad 0\le\theta\le\frac\pi2,
\]
is an $SO(V)$-equivariant homeomorphism. Hence
\[
 S(\Lambda^2\R^4)\cong S^2*S^2\cong S^5.
\]
Under this homeomorphism:
\begin{enumerate}[label=\textup{(\roman*)}]
 \item nonzero rank-two (simple) bivectors form the midpoint
 \[
   \theta=\frac\pi4,\qquad S_+\times S_-\cong S^2\times S^2;
 \]
 \item bivectors with two equal positive singular values form the two join ends
 \[
   \theta=0\quad\text{or}\quad\theta=\frac\pi2,
   \qquad S_+\sqcup S_-\cong S^2\sqcup S^2;
 \]
 \item bivectors with two distinct positive singular values form the two open chambers
 \[
  S^2\times S^2\times\left(0,\frac\pi4\right)
  \quad\sqcup\quad
  S^2\times S^2\times\left(\frac\pi4,\frac\pi2\right).
 \]
\end{enumerate}
\end{theorem}

\begin{proof}
The first statement is the standard polar-coordinate homeomorphism from the join of the unit spheres of two orthogonal vector spaces to the unit sphere of their direct sum. It is equivariant because the Hodge decomposition is $SO(V)$-invariant.

Write a unit bivector as $b=x+y$ with $x\in\Lambda^2_+$ and $y\in\Lambda^2_-$. Then
\[
 b\wedge b=(|x|^2-|y|^2)\vol.
\]
In four dimensions a nonzero bivector is simple exactly when $b\wedge b=0$, hence when $|x|=|y|$, which is $\theta=\pi/4$.

For the singular-value description with the orientation of $V$ fixed, choose an oriented orthonormal basis so that
\[
 b=r_1e_{12}+\varepsilon r_2e_{34},\qquad
 r_1\ge r_2\ge0,\qquad \varepsilon\in\{+1,-1\}.
\]
The singular values are $r_1,r_2$, while the sign $\varepsilon$ records the sign of the Pfaffian. With normalized self-dual and anti-self-dual bases, for $\varepsilon=+1$ one has
\[
 |x|=\frac{r_1+r_2}{\sqrt2},\qquad
 |y|=\frac{r_1-r_2}{\sqrt2},
\]
and for $\varepsilon=-1$ the two norms are interchanged. Thus $r_2=0$ exactly when $|x|=|y|$; $r_1=r_2>0$ exactly when one of $x,y$ vanishes; all remaining nonzero points have both components nonzero and unequal norms.
\end{proof}

\begin{remark}[Orbit spaces]
The midpoint $S^2\times S^2$ is the oriented Grassmannian of $2$-planes in $\R^4$. The two ends are the two components of the space of orthogonal complex structures. The theorem also records the orbit-type closure relations: the generic product fibres collapse one $S^2$ factor at each join end.
\end{remark}

\section{Resolution of Sommen's bivector bill}
\label{sec:resolution}

Sommen partitions $\Lambda^2\R^4\setminus\{0\}$ into three cases: distinct positive singular values, rank two, and equal positive singular values. After restricting to unit norm, \Cref{thm:bivjoin} identifies these cases with the two open chambers, the midpoint, and the two ends of $S^2*S^2$. Hence the partition agrees with the orbit-type decomposition after unit normalization.

Let $M_2(t)=\Mpol_{S^2}(t)$ for any chosen finite cell structure on $S^2$. Since the join is the double mapping cylinder of the two projections $S^2\times S^2\to S^2$, \Cref{thm:dmc} gives the coarse orbit inventory
\[
  \Mpol_{\mathrm{coarse}}(t)=tM_2(t)^2+2M_2(t).
\]
The projection maps, not merely the displayed polynomial, make this a chain model of $S^5$. For the minimal sphere structure $M_2=1+t^2$,
\[
 \Mpol_{\mathrm{coarse}}
 =1+t^5+(1+t)(1+2t^2),
\]
so already the unsplit orbit inventory contains one $E_1$ and two $E_3$ pairs.

After the rank-two midpoint is declared a separate stratum, the same space has inventory
\[
  \Mpol_{\mathrm{split}}(t)=2tM_2(t)^2+M_2(t)^2+2M_2(t).
\]
Thus
\[
  \Mpol_{\mathrm{split}}-\Mpol_{\mathrm{coarse}}
  =(1+t)M_2(t)^2.
\]
The new summand is the additional cell inventory introduced by separating Sommen's rank-two case from the two generic chambers; the two join ends are the equal-singular-value cases.

\begin{theorem}[The excess is join-subdivision overhead]
\label{thm:bill}
Let $J=S^2*S^2$ carry the unsplit join cell structure, let $J'$ be obtained by inserting the rank-two locus at $\theta=\pi/4$, and let $s:C_*(J)\to C_*(J')$ be the subdivision map. Then
\[
 C_*(J')/sC_*(J)\simeq C_*(S^2\times S^2)\otimes E_1,
\]
so
\[
 \Mpol_{J'}(t)-\Mpol_J(t)
 =(1+t)\Mpol_{S^2}(t)^2.
\]
The relative complex is contractible, all of its Smith labels are $[1]$, and
\[
 \Ppol_{J'}(t)=\Ppol_J(t)=1+t^5.
\]
\end{theorem}

\begin{proof}
The rank-two locus is the midpoint subdivision of the join coordinate. Apply \Cref{thm:joinref} with $X=Y=S^2$. The Poincar\'e polynomial follows from \Cref{prop:joinP}.
\end{proof}

\begin{remark}[Compact normalization and the radial factor]
\label{rem:radial}
The theorem is a statement about the compact unit sphere and finite ordinary cellular chains. Sommen's original calculation concerns $\Lambda^2\R^4\setminus\{0\}\cong(0,\infty)\times S^5$ and treats the open radial factor $R_+$ as an independent morphological quantity. An open cell by itself is not a finite CW complex whose ordinary cellular inventory is the monomial $t$. Thus the radial homeomorphism transports the orbit partition, but it does not justify multiplying the finite ordinary-chain polynomial by $t$. A literal realization of Sommen's noncompact bills requires locally finite or Borel--Moore chains, compactly supported invariants, or a constructible Grothendieck theory. The result above identifies the angular source of the discrepancy and does not claim a realization of every subtraction or division used in the original calculus.
\end{remark}

\begin{corollary}[Diagnosis of the inconsistency]
The mismatch in Sommen's $n=4$ calculation is not caused by an incorrect singular-value partition, a missing orbit, or varying stabilizer on a stated stratum. It comes from two related suppressions:
\begin{enumerate}[label=\textup{(\roman*)}]
 \item the calculation compares a midpoint-subdivided inventory of the join cylinder with an unsubdivided quantity for the ambient sphere;
 \item ordinary multiplication records $S^2\times S^2\times(0,1)$ but not the collapse of one sphere factor at each join endpoint.
\end{enumerate}
The apparent excess is precisely a multiple of $1+t$, hence it vanishes under homology and Euler evaluation but remains visible at the raw cell-count level.
\end{corollary}

\begin{remark}[Comparison with Sommen's calculation]
The Hodge decomposition and the homeomorphism $S^2*S^2\cong S^5$ are standard. Here they identify Sommen's three strata with a single join refinement, whose relative complex is the unit-labelled complex of \Cref{sec:refinement}. This accounts for the angular part of the discrepancy in \cite{Sommen2016}. The noncompact radial factor remains outside the finite ordinary-chain model; see \Cref{rem:radial}.
\end{remark}

\section{The five-dimensional bivector sphere}
\label{sec:bivector5}

In dimension five the Hodge decomposition used above is no longer available, and the orbit geometry is different. Identify $\Lambda^2\R^5$ with $\mathfrak{so}(5)$ using the Euclidean metric. The $SO(5)$ action is the adjoint action.

\begin{theorem}[Cohomogeneity-one orbit diagram]
\label{thm:so5diagram}
Every unit bivector is $SO(5)$-conjugate to
\[
 b_\theta=\cos\theta\,e_{12}+\sin\theta\,e_{34},
 \qquad 0\le\theta\le\frac\pi4.
\]
The orbit space of $S(\Lambda^2\R^5)=S^9$ is this interval. Its group diagram is
\[
 T^2\subset
 \left\{SO(2)\times SO(3),\ U(2)\right\}
 \subset SO(5).
\]
Thus the principal and singular orbits are
\begin{align*}
 P&=SO(5)/T^2,\\
 B_-&=SO(5)/(SO(2)\times SO(3))\cong\widetilde G_2(\R^5),\\
 B_+&=SO(5)/U(2)\cong\mathbb{CP}^3.
\end{align*}
Here $\widetilde G_2(\R^5)$ denotes the oriented Grassmannian of $2$-planes in $\R^5$.
Moreover $K_\pm/T^2\cong S^2$, so both singular orbits have codimension three and
\[
 S^9\cong SO(5)\times_{SO(2)\times SO(3)}D^3
 \ \cup_{SO(5)/T^2}\
 SO(5)\times_{U(2)}D^3.
\]
\end{theorem}

\begin{proof}
The orthogonal normal form of a real skew-symmetric $5\times5$ matrix consists of two $2\times2$ rotation blocks and one zero block. Unit norm and the type-$B_2$ Weyl group reduce the parameters to the stated chamber. In its interior the stabilizer is the maximal torus $T^2$. At $\theta=0$ the three-dimensional kernel can be rotated, giving $SO(2)\times SO(3)$. At $\theta=\pi/4$ the two rotation blocks have equal speed and the stabilizer on their four-plane enlarges to $U(2)$. The homogeneous identifications and the isoparametric interpretation are standard; see \cite{QianTangYan2023}. Since both quotient spheres $K_\pm/T^2$ are $S^2$, \Cref{cor:cohom1} gives the disk-bundle decomposition.
\end{proof}

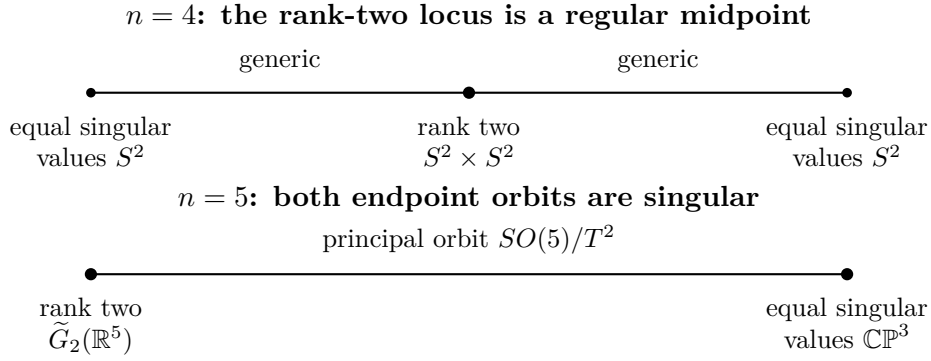
\begin{figure}[t]
\centering
\begin{tikzpicture}[x=1cm,y=1cm,font=\small]
  \node[font=\bfseries] at (5,3.55) {$n=4$: the rank-two locus is a regular midpoint};
  \draw[thick] (0,2.55)--(10,2.55);
  \fill (0,2.55) circle (1.8pt);
  \fill (5,2.55) circle (2.3pt);
  \fill (10,2.55) circle (1.8pt);
  \node[below=5pt,align=center] at (0,2.55) {equal singular\\values $S^2$};
  \node[below=5pt,align=center] at (5,2.55) {rank two\\$S^2\times S^2$};
  \node[below=5pt,align=center] at (10,2.55) {equal singular\\values $S^2$};
  \node[above=4pt] at (2.5,2.55) {generic};
  \node[above=4pt] at (7.5,2.55) {generic};

  \node[font=\bfseries] at (5,1.15) {$n=5$: both endpoint orbits are singular};
  \draw[thick] (0,0.15)--(10,0.15);
  \fill (0,0.15) circle (2.3pt);
  \fill (10,0.15) circle (2.3pt);
  \node[below=5pt,align=center] at (0,0.15) {rank two\\$\widetilde G_2(\R^5)$};
  \node[below=5pt,align=center] at (10,0.15) {equal singular\\values $\mathbb{CP}^3$};
  \node[above=4pt,align=center] at (5,0.15) {principal orbit $SO(5)/T^2$};
\end{tikzpicture}
\caption{The orbit-interval distinction behind the two bivector bills. In dimension four, the rank-two locus is a regular midpoint of the join parameter. In dimension five, the rank-two and equal-singular-value loci are the two singular endpoints.}
\label{fig:orbitintervals}
\end{figure}

The two endpoint strata are Sommen's rank-two and equal-singular-value cases. Unlike the $n=4$ rank-two locus, they are singular orbits of the cohomogeneity-one action rather than strata introduced by subdividing the orbit interval.

\begin{theorem}[Orbit-stratum bill for standard Bruhat models]
\label{thm:n5bill}
Choose the standard Bruhat CW structures for the full and partial flag manifolds in \Cref{thm:so5diagram}, and replace the homogeneous projections by cellular approximations. Put
\[
 Q(t)=1+t^2+t^4+t^6,
 \qquad
 F(t)=1+2t^2+2t^4+2t^6+t^8.
\]
Then
\[
 \Mpol_{B_-}=\Mpol_{B_+}=Q,
 \qquad
 \Mpol_P=F.
\]
The cellular approximations determine a double mapping cylinder homotopy equivalent to $S^9$, and its inventory is
\begin{align}
 \Mpol_{\mathrm{orb}}(t)
 &=2Q(t)+tF(t)\label{eq:n5orb}\\
 &=1+t^9+(1+t)(1+2t^2+2t^4+2t^6).\label{eq:n5normal}
\end{align}
The resulting mapping-cone complex has integral elementary decomposition
\[
 C_{\mathrm{orb}}
 \cong
 \Z[0]\oplus\Z[9]
 \oplus E_1
 \oplus E_3^{\oplus2}
 \oplus E_5^{\oplus2}
 \oplus E_7^{\oplus2}.
\]
In particular, for these Bruhat models all seven excess pairs are unit-labelled and the homological quantity is $1+t^9$.
\end{theorem}

\begin{proof}
The partial flag manifolds $\mathbb{CP}^3$ and $\widetilde G_2(\R^5)$ have Bruhat decompositions with one cell in each of the real dimensions $0,2,4,6$. The full flag manifold $SO(5)/T^2$ has cells indexed by the Weyl group of type $B_2$; its length distribution is $1,2,2,2,1$, giving $F(t)$. By the cellular approximation theorem, the two homogeneous projections may be replaced by cellular maps homotopic to them. Homotopic attaching maps give homotopy-equivalent double mapping cylinders. Therefore \Cref{thm:dmc} gives \eqref{eq:n5orb}, and direct expansion gives \eqref{eq:n5normal}.

The resulting mapping cone is the cellular chain complex of a finite CW complex homotopy equivalent to $S^9$. Hence its integral homology is free of rank one in degrees $0$ and $9$ and zero otherwise. Applying \Cref{thm:smith} to the cell ranks in \eqref{eq:n5normal} forces one unit pair in degrees $(1,0)$, two in $(3,2)$, two in $(5,4)$, and two in $(7,6)$.
\end{proof}

\begin{corollary}[A smaller Thom bill]
A Thom-adapted disk-bundle CW structure has inventory
\[
 \Mpol_{\mathrm{Thom}}(t)
 =Q(t)+t^3Q(t)
 =1+t^9+(1+t)(t^2+t^4+t^6),
\]
and chain decomposition
\[
 C_{\mathrm{Thom}}
 \cong \Z[0]\oplus\Z[9]\oplus E_3\oplus E_5\oplus E_7.
\]
Hence two cellular models of the same cohomogeneity-one geometry may differ by additional unit-labelled pairs.
\end{corollary}

\begin{proof}
Start with one rank-three disk bundle, which retracts to its singular orbit, and attach the other relative to the common sphere bundle. The Thom cells are shifted upward by three. The displayed normal form follows from \Cref{thm:smith} and the homology of $S^9$.
\end{proof}

\begin{corollary}[Diagnosis of Sommen's $n=5$ bill]
After unit-sphere normalization, Sommen's three geometric cases are the orbit types of the $SO(5)$ action, but the scalar sum of their product quantities does not determine the global decomposition. The additional datum is the span
\[
 \widetilde G_2(\R^5)
 \xleftarrow{\ p_-\ }
 SO(5)/T^2
 \xrightarrow{\ p_+\ }
 \mathbb{CP}^3.
\]
The mapping-cone differential contains the two endpoint projections. Suppressing them turns a double mapping cylinder into an ordinary sum and product and produces the apparent angular excess. As in \Cref{rem:radial}, this finite statement does not by itself realize Sommen's open radial factor.
\end{corollary}

\begin{remark}[Why the two dimensions differ]
For $n=4$, the rank-two locus lies in the interior of the orbit interval and its separation creates a subdivision pair. For $n=5$, the rank-two locus is a singular endpoint. The correction is therefore not a midpoint subdivision but the mapping-cone differential of the cohomogeneity-one diagram. This distinction is invisible in the original scalar notation.
\end{remark}

\section{A partial six-dimensional bivector calculation}
\label{sec:bivector6}

The next case is already higher rank but can still be described without a full general theory. Identify $\Lambda^2\R^6$ with $\mathfrak{so}(6)$. The adjoint representation is polar, with a Cartan subalgebra as section and the Weyl group of type $D_3\cong A_3$ acting on it \cite{DadokKac1985,HsiangPalaisTerng1988}.

\begin{proposition}[The spherical $D_3$ chamber]
\label{prop:so6chamber}
Every unit bivector is $SO(6)$-conjugate to
\[
 b_{\lambda}=\lambda_1e_{12}+\lambda_2e_{34}+\lambda_3e_{56},
\]
with
\[
 \lambda_1\ge\lambda_2\ge |\lambda_3|,
 \qquad
 \lambda_1^2+\lambda_2^2+\lambda_3^2=1.
\]
The orbit space of $S(\Lambda^2\R^6)=S^{14}$ is therefore a spherical triangle $\Delta$. The isotropy groups by open face are:
\begin{center}
\begin{tabular}{lll}
\toprule
face & stabilizer & orbit \\
\midrule
interior & $T^3$ & $SO(6)/T^3$ \\
any open edge & $U(2)\times SO(2)$ & $SO(6)/(U(2)\times SO(2))$ \\
$\lambda_1=\lambda_2=\lambda_3$ & $U(3)$ & $SO(6)/U(3)\cong\mathbb{CP}^3$ \\
$\lambda_1=\lambda_2=-\lambda_3$ & $U(3)$ & $SO(6)/U(3)\cong\mathbb{CP}^3$ \\
$\lambda_2=\lambda_3=0$ & $SO(2)\times SO(4)$ & $\widetilde G_2(\R^6)$ \\
\bottomrule
\end{tabular}
\end{center}
The two $U(3)$ vertices are distinct $SO(6)$-orbits, distinguished by the sign of the Pfaffian.
\end{proposition}

\begin{proof}
The orthogonal normal form of a real skew-symmetric $6\times6$ matrix has three $2\times2$ rotation blocks. The type-$D_3$ Weyl group permutes the three block parameters and changes an even number of their signs, giving the displayed chamber. For a regular point, only independent rotations in the three coordinate $2$-planes commute with $b_\lambda$, hence the stabilizer is $T^3$. Equality of two adjacent chamber coordinates combines the corresponding two complex lines and enlarges the centralizer to $U(2)\times SO(2)$. At the two equal-magnitude vertices the centralizer is $U(3)$; at the rank-two vertex the kernel is four-dimensional, giving $SO(2)\times SO(4)$. The quotient $SO(6)/U(3)$ is $\mathbb{CP}^3$, while $SO(6)/(SO(2)\times SO(4))$ is the oriented Grassmannian of $2$-planes.
\end{proof}

The chamber determines a finite face-poset diagram: whenever a face $F'$ lies in the boundary of $F$, one has an inclusion $H_F\subset H_{F'}$ and a homogeneous projection
\[
 SO(6)/H_F\longrightarrow SO(6)/H_{F'}.
\]
This is the two-dimensional analogue of the span used in the $n=5$ double mapping cylinder. The standard polar reconstruction writes the sphere as $(SO(6)\times\Delta)/\!\sim$, where points over a face are identified by its isotropy subgroup. Turning this quotient into a compatible cellular bar model is the remaining attachment problem.

\begin{proposition}[Bruhat face inventory]
\label{prop:n6inventory}
Put $q=t^2$. The standard Bruhat polynomials of the orbit types in \Cref{prop:so6chamber} are
\begin{align*}
 F(t)&=(1+q)(1+q+q^2)(1+q+q^2+q^3), && SO(6)/T^3,\\
 E(t)&=(1+q+q^2)(1+q+q^2+q^3), && SO(6)/(U(2)\times SO(2)),\\
 Q(t)&=1+q+q^2+q^3, && SO(6)/U(3),\\
 R(t)&=1+q+2q^2+q^3+q^4, && \widetilde G_2(\R^6).
\end{align*}
If one records the open two-dimensional chamber interior, its three open edges, and its three vertices formally, the resulting face inventory is
\[
 \mathcal I_6(t)=t^2F(t)+3tE(t)+2Q(t)+R(t).
\]
It satisfies the exact algebraic identity
\[
 \mathcal I_6(t)=1+t^{14}+(1+t)\mathcal B_6(t),
\]
where
\[
 \mathcal B_6(t)=
 2+t+3t^2+3t^3+4t^4+5t^5+3t^6+6t^7+t^8+5t^9+3t^{11}.
\]
In particular, this six-dimensional higher-rank bill passes the necessary nonnegativity test imposed by \Cref{thm:normal}.
\end{proposition}

\begin{proof}
The full flag polynomial is the Weyl length polynomial of type $D_3\cong A_3$, whose exponents are $1,2,3$. Passing to an edge quotient divides by the length polynomial $1+q$ of the corresponding $A_1$ Weyl subgroup. At a $U(3)$ vertex the Weyl subgroup is $A_2$, giving $Q$; at the rank-two vertex the Weyl subgroup is $A_1\times A_1$, giving $R$. The displayed identity is direct polynomial expansion.
\end{proof}

\begin{remark}[What is and is not proved in dimension six]
The chamber, stabilizers, homogeneous projections, and Bruhat polynomials are determined. The coefficients of the correction polynomial are not identified with boundary ranks, because this would require a single cellular model in which all face projections are compatible and whose bar realization is verified against $S^{14}$. The $n=6$ calculation therefore provides a finite test case for a general Weyl-chamber construction, while leaving the chain-level compatibility problem open.
\end{remark}

\section{Relation with canonical and formal styles}
\label{sec:styles}

At the canonical commutative layer, the dictionary is:
\begin{center}
\begin{tabular}{>{\raggedright\arraybackslash}p{0.31\textwidth} >{\raggedright\arraybackslash}p{0.59\textwidth}}
\toprule
Morphological notation & Script/topological interpretation \\
\midrule
$a_0R^n+\cdots+a_n$ & cell-count polynomial $\Mpol_S(t)$ \\
addition & direct sum or disjoint union \\
multiplication as product & tensor product of reduced chain complexes \\
$R^k\sim2R^k+R^{k-1}$ & unit elementary pair $E_k=T_k(1)$ \\
integer-labelled pair & $T_k(d)$; $d>1$ creates torsion \\
substitution $R=-1$ & quotient by elementary pairs, giving Euler characteristic \\
``same quantity'' & same graded cell ranks, possibly different differentials \\
bundle over $S^1$ & mapping torus, corrected by $1-f_*$ on homology \\
$X/G$ & action-certified quotient $(X,G,\rho)$; not scalar inversion \\
open radial cell $R_+$ & Borel--Moore/relative degree-one generator in the selected noncompact models \\
sphere addition formula & join, not ordinary Cartesian multiplication \\
cohomogeneity-one partition & double mapping cylinder / mapping cone of the endpoint projections \\
\bottomrule
\end{tabular}
\end{center}

The formal style retains order and parentheses and cannot be recovered from a commutative polynomial. It would require a free non-symmetric monoidal or operadic syntax followed by realization into scripts.

\section{Scope, limitations, and further directions}
\label{sec:limits}

The preceding constructions address four parts of Sommen's language: finite cell inventories, finite attachment diagrams, selected noncompact conical models, and action-certified quotients. None of these makes the reduced script complex a complete invariant of a space. For example,
\[
 C_*(\mathbb{CP}^2;\mathbb Z)\cong C_*(S^2\vee S^4;\mathbb Z)
 \cong \mathbb Z[0]\oplus\mathbb Z[2]\oplus\mathbb Z[4]
\]
with zero differential, whereas the spaces are not homotopy equivalent: the square of the degree-two cohomology class is nonzero for $\mathbb{CP}^2$ and zero for the wedge. This example shows that cell counts and differentials do not retain nonlinear, homotopy-coherent multiplicative structure.

\subsection*{Further extensions}

\paragraph{Refinement.}
\Cref{thm:generalrefinement} gives a general theorem at the chain level: a quasi-isomorphic refinement has a unit-only mapping-cone defect, and a degreewise split injective one has a unit-only relative defect. What remains open is a geometric criterion, expressed in the language of tight scripts and geomaps, guaranteeing that a given refinement is a quasi-isomorphism and is degreewise primitive (equivalently split in the present finite free setting). The published notion of refinement gives injectivity but does not by itself settle this question \cite{CerejeirasKahlerLegatiuk2025}.

\paragraph{Higher bivectors.}
For $n=6$, \Cref{sec:bivector6} determines the spherical-triangle chamber, its face stabilizers, and a Bruhat face inventory with a nonnegative $(1+t)$ correction. Compatibility of the face projections at chain level remains unresolved. A general $n\ge6$ theory should construct a cellular or script-level homotopy-colimit model for the face poset of the type-$B$ or type-$D$ spherical chamber.

\paragraph{Noncompact strata.}
\Cref{sec:borelmoore} gives Borel--Moore models for the line, punctured Euclidean space, and Sommen's conical null-cone bill. A general theory would require a category of locally compact stratified scripts with proper maps, excision for closed/open decompositions, products, and a controlled realization of complements. Such a theory is not developed here.

\subsection*{Open problem: homotopy-coherent structure}

A further limitation concerns classification. Ordinary chain complexes do not retain cup products, and ordinary cohomology rings do not retain higher cochain operations. This suggests considering an $E_\infty$ enhancement of script cochains. Mandell proved that finite-type nilpotent spaces are weakly equivalent if and only if their integral singular cochains are quasi-isomorphic as $E_\infty$ algebras \cite{Mandell2006}. In the present setting, this theorem gives a comparison target for any proposed homotopy-complete enhancement of script cochains.

\begin{quote}
\textbf{Open problem.} Construct, for an appropriate category of geometric scripts, a functorial $E_\infty$ cochain object $\mathcal E(S)$ such that:
\begin{enumerate}[label=\textup{(\roman*)}]
 \item its underlying cochain complex is quasi-isomorphic to $\operatorname{Hom}(C(S),\mathbb Z)$;
 \item refinement geomaps that preserve realization induce $E_\infty$ quasi-isomorphisms;
 \item the construction is compatible, up to coherent equivalence, with the finite homotopy-colimit diagrams used in this paper; and
 \item for realizable finite-type nilpotent scripts, $\mathcal E(S)$ agrees with singular $E_\infty$ cochains of the realization.
\end{enumerate}
Determine the weakest geometric hypotheses under which such an enhancement exists and whether it is conservative on homotopy types.
\end{quote}

The problem points to coherent diagonals and higher attaching operations as data not represented by the present chain complexes. We leave their construction to future work.

\section{Conclusion}

Finite morphological quantities can be interpreted as cell-count polynomials of script complexes. Their excess over homology is a sum of adjacent-dimensional pairs; over \(\mathbb Z\), Smith labels distinguish removable unit pairs from torsion-bearing attachments. Cartesian products, bundles, and gluings must be separated: tensor products model the first, monodromy corrects the second, and finite homotopy-colimit diagrams retain the attachment maps in the third.

The four- and five-dimensional bivector calculations exhibit two distinct sources of apparent overcounting. For \(\Lambda^2\mathbb R^4\), the rank-two locus is a regular midpoint of the join \(S^2*S^2\), and isolating it creates subdivision overhead. For \(\Lambda^2\mathbb R^5\), the rank-two locus is a singular endpoint of a cohomogeneity-one action; the global model used here is the double mapping cylinder joining \(\widetilde G_2(\mathbb R^5)\), \(SO(5)/T^2\), and \(\mathbb{CP}^3\). The six-dimensional chamber calculation gives a finite higher-rank test case, and the Borel--Moore examples recover Sommen's selected open radial and conical terms after support conditions are specified. Division is defined when an action or homogeneous-quotient certificate is supplied. The remaining limitation is the absence of homotopy-coherent multiplicative structure, formulated above as an open problem.

\paragraph{Funding.}
This work was co-funded by the University of Ostrava, Grant No.~SGS05/P\v{R}F/2026. Yifan Zhang is also co-funded by the European Union under the REFRESH -- Research Excellence For REgion Sustainability and High-tech Industries project, No.~CZ.10.03.01/00/22003/0000048, via the Operational Programme Just Transition.

\end{document}